\documentclass[12pt]{article}
\usepackage{amssymb,amsthm,amsmath,amsfonts, fullpage, ytableau, shuffle,enumitem}
\usepackage{mycommands}

\usepackage{todonotes} 
\usepackage[final]{showlabels}

\theoremstyle{plain}
\newtheorem{thm}{Theorem}[section]
\newtheorem{prop}[thm]{Proposition}
\newtheorem{corollary}[thm]{Corollary}
\newtheorem{lemma}[thm]{Lemma}

\theoremstyle{definition}
\newtheorem{definition}[thm]{Definition}
\newtheorem{notation}[thm]{Notation}
\newtheorem{definitions}[thm]{Definitions}
\newtheorem{exa}[thm]{Example}

\begin{document}
\pagestyle{plain}

\title{Symmetric multisets of permutations}
\author{Jonathan S. Bloom}
\date{\today\\[10pt]}

\maketitle

\begin{abstract}
The following long-standing problem in combinatorics was first posed in 1993 by Gessel and Reutenauer \cite{GesReut93}. For which multisubsets $B$ of the symmetric group $\fS_n$ is the quasisymmetric function
$$Q(B) = \sum_{\pi \in B}F_{\Des(\pi), n}$$
a symmetric function?  Here $\Des(\pi)$ is the descent set of $\pi$ and  $F_{\Des(\pi), n}$ is Gessel's fundamental basis for the vector space of quasisymmetric functions.  The purpose of this paper is to provide a useful characterization of these multisets. Using this characterization we prove a conjecture of Elizalde and Roichman from~\cite{ElizaldeRoichman2015}.  Two other corollaries are also given.  The first is a short new proof that conjugacy classes are symmetric sets, a well known result first proved by Gessel and Reutenauer~\cite{GesReut93}.  Our second corollary is a unified explanation that both left and right multiplication of symmetric multisets, by inverse $J$-classes, is symmetric.  The case of right multiplication was first proved by Elizalde and Roichman in~\cite{ElizaldeRoichman2015}.  
\end{abstract}

\section{Introduction}\label{sec:intro}
 For integers $m\leq n$ set $[m,n] := \{m,m+1,\ldots, n\}$.  When $m=1$ we simply write $[n]$ instead of $[1,n]$.  We denote by $\fS_n$ the symmetric group on $[n]$. A multiset $B$ whose elements are taken from $\fS_n$ is denoted by $B\Subset \fS_n$.  The standard notation $B\subseteq \fS_n$ is reserved to indicate that $B$ is a set.  Additionally, given two multisets $A,B\Subset \fS_n$ we write $A\sqcup B$ to denote their disjoint union.  

For any $\pi\in \fS_n$ its \emph{descent set} is 
$$\Des(\pi):=\makeset{i\in [n-1]}{\pi_i>\pi_{i+1}}\subseteq [n-1].$$  
For each $J\subseteq[n-1]$ define the \emph{inverse $J$-class} as 
$$R_J^{-1} = \makeset{\pi^{-1}\in \fS_n}{\Des(\pi)\subseteq J}.$$ 
For multisets $A,B\Subset \fS_n$ we write $A\equiv B$ to indicate that there is descent-set-preserving bijection between $A$ and $B$ or, equivalently, in terms of generating functions, that 
$$\sum_{\pi\in A} \textbf{x}^{\Des(\pi)} = \sum_{\pi\in B} \textbf{x}^{\Des(\pi)}.$$
To indicate that $\Des(\pi) = \Des(\tau)$ for $\pi,\tau\in\fS_n$ we abuse this notation and sometimes write $\pi\equiv \tau$. Further, define $AB$ to be the multiset of all products $\pi\tau$ where $\pi\in A$ and $\tau\in B$.  In the case when $A= \{\pi\}$ we simply write $\pi B$.   Playing a crucial role in this paper are the products $BR_J^{-1}$ and $R_J^{-1}B$.  In the case that $B$ is such that 
$$BR_J^{-1}\equiv R_J^{-1}B$$
for all $J\subseteq[n-1]$ we say that $B$ is \emph{$D$-commutative with all inverse $J$-classes}.

Next we establish some basic notions for the theory of symmetric functions.  Let $\mathbf{x}= \{x_1,x_2,\ldots\}$ be a countably infinite set of commuting variables. We say a formal power series in $\mathbb{Q}[[\mathbf{x}]]$ is \emph{symmetric} if it is of bounded degree and invariant under all permutations of its indices.  The vector space of all homogeneous symmetric functions of degree $n$ is denoted by $\Lambda(n)$.  We make use of two classical bases for $\Lambda(n)$.  Given an integer partition $\lambda = (\lambda_1 \geq \lambda_2\geq \cdots\geq \lambda_p )\vdash n$ we set $m_\lambda$ to be the symmetric function obtained by symmetrizing the monomial 
$$x_{1}^{\lambda_1}\cdots x_{p}^{\lambda_p}.$$
The $m_\lambda$ are called \emph{monomial symmetric functions} and they constitute our first basis.  We also need the \emph{Schur functions} which we denote by $\{s_\lambda\}_{\lambda\vdash n}$.  For a combinatorial definition of these functions and a proof that they are a basis for $\Lambda(n)$ we direct the reader to \cite{FultonBook} or \cite{EC2}.  We say that $f\in \Lambda(n)$ is \emph{Schur positive} provided that it can be written as
$$f= \sum_{\lambda\vdash n} c_\lambda s_\lambda$$
where $c_\lambda$ are nonnegative integer coefficients. 

We next recall the quasisymmetric functions.  For each integer composition $\alpha=(\alpha_1,\ldots, \alpha_p)\models n$  we define the \emph{monomial symmetric function} as
$$M_\alpha: = \sum_{i_1<i_2<\cdots <i_p} x_{i_1}^{\alpha_1}x_{i_2}^{\alpha_2}\cdots x_{i_p}^{\alpha_p}.$$
We denote by $\QSYM(n)$ the $\mathbb{Q}$-span of $\{M_\alpha\}_{\alpha\models n}$ and call the elements of this space  \emph{quasisymmetric functions}.  As the monomial symmetric functions, which are indexed by compositions of $n$,  are a canonical basis for this space, we have $\dim(\QSYM(n)) = 2^{n-1}$.  Lastly, since 
$$m_\lambda = \sum_{\alpha} M_\alpha,$$
where the sum is over compositions $\alpha$ obtained by permuting the parts of $\lambda$, it follows that $\Lambda(n) \subseteq \QSYM(n)$.  

We shall also need Gessel's fundamental basis for $\QSYM(n)$.  To define this basis set
$$F_\alpha = \sum_{\beta\leq \alpha} M_\beta,$$ 
where  $\beta\leq \alpha$ indicates that $\beta$ is a refinement of $\alpha$.  The collection $\{F_\alpha\}_{\alpha\models n}$  is called the \emph{fundamental} basis for $\QSYM(n)$.  

To connect quasisymmetric functions to multisets of permutations we recall the well-worn bijection between subsets $J = \{j_1<j_2<\cdots < j_s\}$ of $[n-1]$ and compositions of $n$ given by 
$$J \mapsto (j_1, j_2-j_1, \ldots , j_{s}-j_{s-1}, n-j_s).$$ 
We denote the image of $J$ under this bijection by $\co(J)$.  Using this correspondence we also index the fundamental basis of $\QSYM(n)$ by subsets of $[n-1]$ and write $F_{J,n}: = F_\alpha$ where $\alpha$ corresponds to $J\subseteq[n-1]$.  For any $B\Subset \fS_n$ we recall the quasisymmetric function 
$$Q(B) := \sum_{\pi\in B} F_{\Des(\pi),n}\in \QSYM(n),$$
first defined in \cite{Ges83}. 
We say that $B$ is \emph{symmetric} if $Q(B)$ is symmetric.  If, moreover, $Q(B)$ is  Schur-positive then, following the language first established by Adin and Roichman in \cite{AdinRoichman15}, we say $B$ is \emph{fine}.

Writing $J\sim K$ whenever $\co(J)$ is a permutation of $\co(K)$, the main result of this paper is a proof that the following are equivalent: 
\begin{enumerate}
	\item[a)] $B$ is $D$-symmetric (defined in Section~\ref{sec:char})
	\item[b)] $B$ is $D$-commutative with all $J$-classes
	\item[c)] $BR_J^{-1}\equiv BR_K^{-1}$ whenever $J\sim K$
	\item[d)] $R_J^{-1}B\equiv R_K^{-1}B$ whenever $J\sim K$
	\item[e)] $B$ is symmetric.
\end{enumerate}

As an immediate consequence we prove Conjecture~10.4 in \cite{ElizaldeRoichman2015}, due to Elizalde and Roichman, in which they hypothesizes that if $B$ is fine then 
$$BD_J^{-1} \equiv D_J^{-1}B$$
where $D_J^{-1} = \makeset{\pi^{-1}\in \fS_n}{\Des(\pi) = J}$.  Two additional corollaries are proved.  The first gives a new proof of the fact that conjugacy classes are symmetric.  This was first established by Gessel and Reutenauer in~\cite{Ges83} in which they proved the stronger result, by way of a more involved proof, that  conjugacy classes are actually fine sets.  Our second corollary gives a unified explanation that the multiset obtained by either left or right multiplication of symmetric multisets by $R_J^{-1}$ is symmetric.  The case of right multiplication in the context of fine sets was first proved by Elizalde and Roichman in~\cite{ElizaldeRoichman2015} although their techniques were unable to establish the case of left multiplication. 

The paper is organized as follows.  In the next section we establish key definitions and lemmas used throughout.  In Section~\ref{sec:char} we formally state our main theorem and prove the aforementioned corollaries.  As the proof of our main theorem is involved, we break it into several propositions.  The propositions that follow immediately from definitions are in the main part of Section~\ref{sec:char}.  Statements that imply symmetry are contained in Subsection~\ref{subsec:sufficienty} as they involve a careful analysis of the change of basis between the fundamental basis and the monomial quasisymmetric functions.   The proof that symmetry implies $D$-symmetry involves techniques from the theory of tableaux and occupies Subsection~\ref{subsec:Knuth}.  

\section{Preliminaries}

An \emph{ordered set partition} of $[n]$ is a sequence $\cU=(U_1,\ldots, U_s)$ of nonempty disjoint sets $U_i$, called \emph{blocks}, whose union is $[n]$. We write $\cU\vdash [n]$ to indicate that $\cU$ is an ordered set partition of $[n]$ and we denote by $\Pi(n)$ the set of all ordered set partitions of $[n]$.  For each composition $\alpha\models n$ we further define $\Pi(\alpha)$ to be the collection of all $\cU\vdash [n]$ where $|U_i| = \alpha_i$. If $\alpha$ corresponds to $J\subseteq [n-1]$ then we also denote this set by $\Pi(n,J)$.  

\begin{exa}
The set partitions
$$\cU=(\{2,5,8\},\{1,3,4\},\{6,7,9\})\quad\myand\quad \cV=(\{2,5,8\},\{6,7,9\},\{1,3,4\})$$
are distinct ordered with $\cU,\cV\in \Pi(3,3,3)= \Pi(\{3,6\},9)$.	
\end{exa}

For the remainder of this section fix $J=\{j_1<\cdots <j_s\}\subseteq [n-1]$ and adopt the convention that $j_0=0$ and $j_{s+1} = n$.  To avoid repeating ourselves the symbol $\cU$ will always denote an element of $\Pi(n,J)$ and $U_i$ will always denote its $i$th part.

\begin{definition}
Let $U$ and $V$ be disjoint alphabets and let $\pi$ and $\tau$ be permutations of $U$ and $V$ respectively.  We define the  \emph{shuffle} of $\pi$ and $\tau$ to be the set $\pi\shuffle \tau$ consisting of all permutations of $U\cup V$ where the letters in $U$ appear in the same order as in $\pi$ and the letters in $V$ appear in the same order as in $\tau$.  

Additionally, for $\pi\in\fS_n$ and $\tau\in\fS_m$ we define
$$\pi\cshuffle \tau : = \pi \shuffle \tau^{+n}$$
where $\tau^{+n}$ is the word obtained by adding $n$ to each term of $\tau$.  So $\pi\cshuffle \tau \subseteq \fS_{n+m}$.  
\end{definition}

\begin{exa}
If $\pi = 12$ and $\tau = 21$ then 
$$\pi\cshuffle \tau = \{ 1243, 1423, 4123, 1432, 4132, 4312\},$$
where we have dropped commas for readability. 	
\end{exa}

As the main results in this paper involve the sets $R_J^{-1}$ we now look at multiple ways of describing such sets. First note that
\begin{equation}\label{eq:R_J as shuf}
R_J^{-1} =  (1\ldots j_1) \shuffle  (j_1+1\ldots j_2) \shuffle \cdots \shuffle (j_s+1\ldots n).
\end{equation}
Hence to every $\cU\in \Pi(n,J)$ there corresponds the permutation $\delta_\cU \in R_J^{-1}$ defined so that $U_1$ is the set of positions occupied by the subsequence  $(1,2,\ldots,j_1)$, $U_2$ is the set of positions occupied by the subsequence $(j_1+1,\ldots, j_2)$, etc. So 
\begin{equation}\label{eq:R_J as delta}
R_J^{-1} = \makeset{\delta_\cU}{\cU\in \Pi(n,J)}
\end{equation}
from which it follows that $\Pi(n,J)$ is in correspondence with  $R_J^{-1}$. 

\begin{exa}
If $n=4$ and $J = \{2\}\subseteq [3]$ then we obtain the following correspondence: 
\begin{center}
\begin{tabular}{ccc}
$\cU$ & $\to$ & $\delta_\cU$\\
\hline
(12, 34)&& 1234\\
(13, 24)&& 1324\\
(14, 23)&& 1342\\
(23, 14)&& 3124\\
(24, 13)&& 3142\\
(34, 12)&& 3412
\end{tabular}	
\end{center}
where, for example, we have written $(12, 34)$ instead of the more verbose $(\{1,2\}, \{3,4\})$.  	
\end{exa}

Next we consider right multiplication by inverse $J$-classes. Recall that if  $\pi,\tau\in \fS_n$ then 
$$\pi\cdot (\tau_1\tau_2\ldots\tau_n) = \pi_{\tau_1}\ldots \pi_{\tau_n},$$ 
where $\cdot$ denotes group multiplication.    Applying this to (\ref{eq:R_J as shuf}) it follows that 
\begin{equation}\label{eq:shuffle pi}
\pi  R_J^{-1} = (\pi_1\ldots \pi_{j_1}) \shuffle(\pi_{j_1+1}\ldots \pi_{j_2})\shuffle  \cdots \shuffle (\pi_{j_s+1}\ldots \pi_{n}).
\end{equation}

\begin{exa}
When $\cU =(\{2,5,8\},\{1,3,4\},\{6,7,9\})$ we see that
$$\delta_\cU = {\color{blue}4}\ {\color{red}1}\ {\color{blue}5}\ {\color{blue}6}\ {\color{red}2}\ 7\ 8\ {\color{red}3}\ 9.$$
and if we take $\pi = {\color{red} 9\ 1\ 4\ }{\color{blue} 5\ 2\ 8\ }7\ 3\ 6$
then 
$$\pi\cdot\delta_\cU = {\color{blue}5}\ {\color{red}9}\ {\color{blue}2}\ {\color{blue}8}\ {\color{red}1}\ 7\ 3\ {\color{red}4}\ 6.$$
Note that the first three terms in $\pi$ (colored red) are in positions $2,5,8$, the next three (colored blue) are in positions $1,3,4$ and the last three (black) are in positions $6,7,9$. 	
\end{exa}

We now require the following consequence of Stanley's famous shuffling theorem (see, \cite[Exercise~3.161]{EC1}).

\begin{lemma}\label{lem:shuf invariant}
Assume $\pi$ and $\tau$ are permutations of disjoint alphabets as are $\pi'$ and $\tau'$.   If $\pi\equiv \pi'$ and $\tau \equiv \tau'$ then
$$\pi\shuffle \tau \equiv \pi'\shuffle \tau'.$$
\end{lemma}

Using this lemma together with (\ref{eq:shuffle pi}) we see that
\begin{equation}\label{eq:R_J as cshuf}
\pi R_J^{-1} \equiv \std(\pi_1\ldots \pi_{j_1}) 
\cshuffle \std(\pi_{j_{1}+1}\ldots \pi_{j_2})
\cshuffle \cdots 
\cshuffle \std(\pi_{j_s}\ldots \pi_{s+1}),	
\end{equation}
where $\std$ is \emph{standardization}. 

Although not apparent at this point it turns out that it is much easier to work with the elements on the right side of (\ref{eq:R_J as cshuf}).  To do this we make the following definition.

\begin{definition}
For $\cU\in \Pi(n,J)$ and $\pi\in \fS_n$ we define $\sigma_\cU(\pi)$ to be the element on the right side of (\ref{eq:R_J as cshuf}) in which $U_1$ is the set of positions occupied by the subsequence $\std(\pi_1\ldots \pi_{j_1})$, $U_2$ is the set of positions occupied by the subsequence $\std(\pi_{j_1+1}\ldots \pi_{j_2})^{+j_1}$, etc.     
\end{definition}

The next lemma is now immediate and makes use of a standard convention.   For any function $f:\fS_n \to \fS_n$ and $B\Subset\fS_n$ we denote by $f(B)$ the multiset obtained by applying $f$ to each element in $B$.   

\begin{lemma}\label{lem:r as sigma}
For any $B\Subset \fS_n$ we have 
$$BR_J^{-1} \equiv \bigsqcup_{\cU\in \Pi(n,J)} \sigma_\cU(B).$$
\end{lemma}

\medskip

Next consider left multiplication by inverse $J$-classes.  Although it trivially follows by our definitions that 
$$R_J^{-1}\pi = \makeset{\delta_\cU \pi}{\cU\in \Pi(n,J)},$$
a ``twist" is needed in order for us to compare left and right multiplication.  As such we make the following definition.
\begin{definition}
For each $\pi\in \fS_n$ define 
$$\rho_\cU(\pi) := \delta_\cW\cdot \pi$$
where $\cW$ is the ordered set partition in $\Pi(n,J)$ whose $i$th block is $\pi(U_i)$.
\end{definition}
Since the mapping $\pi:\Pi(n,J) \to \Pi(n,J)$ given by 
$$(U_1,U_2,\ldots) \mapsto (\pi(U_1),\pi(U_2),\ldots)$$
is clearly bijective we obtain the next lemma.

\begin{lemma}\label{lem:l as rho}
For any $B\Subset \fS_n$ we have 
$$R_J^{-1}B = \bigsqcup_{\cU\in \Pi(n,J)} \rho_\cU(B).$$
\end{lemma}

We end this section with an important description of how the descent structure of $\rho_\cU(\pi)$ and $\sigma_\cU(\pi)$ are related.  For any $S\subseteq \mathbb{Z}$ set $S^*=\makeset{i\in S}{i+1\in S}$.  Using this define for any $\cU= (U_1,U_2,\ldots) \vdash [n]$ the set
$$\cU^* := U_1^* \cup U_2^* \cup \cdots\subseteq [n-1]$$

\begin{lemma}\label{lem:des char}
Let $\pi,\tau\in \fS_n$ and $\cU\in \Pi(n,J)$. Then $\rho_\cU(\pi) \equiv \sigma_\cU(\tau)$
if and only if for each $u\in \cU^*$ we have
$$u \in \Des(\pi) \iff \delta_\cU(u) \in \Des(\tau).$$
Moreover we have
\begin{equation}\label{eq:des rho}
\Des(\rho_\cU(\pi)) = \Des(\delta_\cU) \cup \left(\Des(\pi)\cap \cU^*\right)	
\end{equation}
and
\begin{equation}\label{eq:des sigma}
\Des(\sigma_\cU(\tau)) = \Des(\delta_\cU) \cup  \makeset{u\in \cU^*}{\delta_\cU(u)\in\Des(\tau)}.	
\end{equation}
\end{lemma}

\begin{proof}
As $\delta_\cU$ is increasing on the blocks of $\cU$, then $\Des(\delta_\cU)$ and $\Des(\pi)\cap\cU^*$ are disjoint.  Hence the two sets on the right side of (\ref{eq:des rho}) and (\ref{eq:des sigma}) are, respectively, disjoint.  Therefore to prove the first claim it suffices to show that (\ref{eq:des rho}) and (\ref{eq:des sigma}) hold. Before continuing we make the following definition.  For finite sets of integers $A$ and $B$ we write $A<B$ provided that $\max(A) < \min(B)$.

We first prove (\ref{eq:des rho}).  Take $\cW\in \Pi(n,J)$ so that its $i$th block is $W_i:=\pi(U_i)$.  By our definition $\rho_\cU(\pi) = \delta_\cW\cdot \pi$.  Observe that $\delta_\cW$ is increasing on each block $W_i$ and $\delta_\cW(W_i)> \delta_\cW(W_j)$ whenever $i>j$.  So  $\delta_\cW(\pi(u)) > \delta_\cW(\pi(u+1))$ 
if and only if either
\begin{itemize}
\item[a)] $\pi(u)> \pi(u+1)$ and $u,u+1\in U_i$, or
\item[b)] $\pi(u) \in W_i$ and $\pi(u+1)\in W_j$ with $i>j$.
\end{itemize}
As the first condition is equivalent to $u\in \Des(\pi)\cap\cU^*$ and the second is equivalent to $u\in U_i$ and $u+1\in U_j$ with $i>j$, i.e., $u\in \Des(\delta_\cU)$, we arrive at (\ref{eq:des rho}).

Next we prove (\ref{eq:des sigma}). From the definitions we see that 
$$\sigma_\cU(\tau)(U_i) > \sigma_\cU(\tau)(U_j) \iff i>j \iff  \delta_\cU(U_i)>\delta_\cU(U_j).$$
If $u,u+1$ are in distinct blocks of $\cU$ then
$$u\in \Des(\sigma_\cU(\tau))\iff u\in\Des(\delta_\cU).$$
Next assume $u,u+1\in U_i$, i.e., $u\in \cU^*$.  Note that the subsequences of $\sigma_\cU(\tau)$ and $\delta_\cU$ indexed by $U_i$ are  
$$\std(\tau_{j_{i-1}+1},\ldots, \tau_{j_i})^{+j_{i-1}}\quad \myand \quad (j_{i-1}+1, j_{i-1}+2, \ldots, j_i),$$
respectively.  In this case we have
$$u\in \Des(\sigma_\cU(\tau))\iff \delta_\cU(u)\in \Des(\tau).$$ 
Together these cases prove that (\ref{eq:des sigma}) holds.

%
%

\end{proof}

\begin{lemma}\label{lem:mult invariance}
	For any $A,B\Subset \fS_n$ if $A\equiv B$ then 
	$$AR_J^{-1} \equiv BR_J^{-1}\quad\myand\quad R_J^{-1}A \equiv R_J^{-1}B.$$
\end{lemma}
\begin{proof}
Let $f:A\to B$ be our descent-set-preserving bijection and let $\pi\in A$. By Equations (\ref{eq:des sigma}) and (\ref{eq:des rho}) from the previous lemma, it follows that $\sigma_\cU(\pi)\equiv \sigma_\cU(f(\pi))$ and $\rho_\cU(\pi)\equiv \rho_\cU(f(\pi))$. This lemma now follow from Lemmas~\ref{lem:r as sigma} and \ref{lem:l as rho}, respectively.  
\end{proof}

\section{A characterization of symmetric sets}\label{sec:char}

The purpose of this section is to state and prove a useful characterization of symmetric multisets. Using this characterization we then simultaneously explain several well known results in the theory of symmetric sets as well as prove the aforementioned conjecture of Elizalde and Roichman.     To state our main theorem we require a couple of definitions. Throughout this section $B\Subset \fS_n$ and $J,K\subseteq[n-1]$. 

\begin{notation}
	Let $\alpha,\beta\models n$.  We write $\alpha \sim \beta$ to indicate that the sequence  $\beta$ is a permutation of the sequence $\alpha$.  For $J, K\subseteq [n-1]$ we also write $J\sim K$ provided that $\co(J)\sim \co(K)$.  
\end{notation}

\begin{definition}
	We say $B\Subset \fS_n$ is \emph{$D$-symmetric} if for every $\cU\vdash[n]$ there exists a bijection $\Psi_\cU:B\to B$ so that 
	for each $u \in \cU^*$ and $\pi \in B$ we have
\begin{equation}\label{eq:Dsym cond}
u\in \Des(\pi) \iff \delta_{\cU}(u) \in \Des(\Psi_\cU(\pi)).
\end{equation} 
\end{definition}

We now come to our main theorem.
\begin{thm}\label{thm:main}
The following are equivalent:
	\begin{enumerate}
		\item[a)] $B$ is $D$-symmetric
		\item[b)] $B$ is $D$-commutative with all inverse $J$-classes
		\item[c)] $BR_J^{-1}\equiv BR_K^{-1}$ whenever $J\sim K$
		\item[d)] $R_J^{-1}B\equiv R_K^{-1}B$ whenever $J\sim K$
		\item[e)] $B$ is symmetric.
	\end{enumerate}
\end{thm}

Before diving into the proof of this theorem we state and prove several corollaries.  

In \cite{ElizaldeRoichman2015} Elizalde and Roichman hypothesis that if $B$ is fine then 
$$BD_J^{-1} \equiv D_J^{-1}B$$
where $D_J^{-1} = \makeset{\pi^{-1}\in \fS_n}{\Des(\pi) = J}$.  By a straightforward application of inclusion-exclusion their conjecture is equivalent to showing that fine sets (which of course are symmetric) $D$-commute with all inverse $J$-classes.  As this is immediate from Theorem~\ref{thm:main} we record it as our first corollary.  

\begin{corollary}\label{cor:conjecture}
If $B\Subset \fS_n$ is fine then $B$ is $D$-commutative with all inverse $J$-classes.  
\end{corollary}

Another immediate corollary of our theorem is the well known fact that conjugacy classes are symmetric.  Gessel and Reutenauer first proved the stronger fact in~\cite{GesReut93} that such sets are actually fine. Their proof relies on ideas from representation theory.  

\begin{corollary}\label{cor:conjugacy}
	Let $C\subseteq\fS_n$ be a conjugacy class. Then $C$ is symmetric.  
\end{corollary}

\begin{proof}
	As $C$ is a conjugacy class we know that $C\pi = \pi C$ for all $\pi\in \fS_n$.  So for any  $S\subseteq\fS_n$ we have $CS = SC$.  In particular $C$ is $D$-commutative with all inverse $J$-classes and our claim follows by Theorem~\ref{thm:main}.  
\end{proof}

Our next corollary simultaneously explains why the collection of symmetric multisets of $\fS_n$ is closed under multiplication by inverse $J$-classes on the right and on the left.  In \cite{ElizaldeRoichman2015} Elizalde and Roichman prove that right multiplication of fine multisets by inverse $J$-classes yields fine multisets.  Although not explicitly done in their paper, one can easily extend this result to conclude that the same holds for symmetric multisets.  Their results again use ideas from representation theory.  That said, they were unable to obtain similar results in the context of left multiplication which, one can speculate, is the reason for their Conjecture 10.4. We now provide a short uniform explanation for symmetric invariance under both left and right multiplication by inverse $J$-classes.

\begin{corollary}
	For any symmetric  $B\Subset \fS_n$ the multisets $R_J^{-1}B$ and $BR_J^{-1}$ are also symmetric.
\end{corollary}
\begin{proof}

Take $B$ as stated and consider $K\sim K'\subseteq [n-1]$.  As $B$ is symmetric Theorem~\ref{thm:main} tells us that 
$$R_K^{-1}B \equiv  R_{K'}^{-1}B\quad\myand\quad BR_K^{-1} \equiv  BR_{K'}^{-1}.$$  
Therefore for any $J\subseteq[n-1]$ it follows from Lemma~\ref{lem:mult invariance} that 
$$R_K^{-1}BR_J^{-1}\equiv R_{K'}^{-1}BR_J^{-1}\quad\myand\quad R_J^{-1}BR_K^{-1}\equiv R_{J}^{-1}BR_{K'}^{-1}.$$
By another application of Theorem~\ref{thm:main} we conclude that $BR_J^{-1}$ and $R_J^{-1}B$ are both symmetric.
\end{proof}

We note that the above corollary does not hold if $R_J^{-1}$ is replaced by an arbitrary symmetric set $A$. For example if $A = \{ 1324, 4132\}$ and $B = \{2143, 2314\}$ then 
$$Q(A) =Q(B)= m_{22} + m_{211} + 2m_{1111}$$
but $Q(AB) =M_{31} +  M_{22} +  2M_{112} + 2M_{121} + 2M_{211} + 4M_{1111}$.

\medskip

We now turn our attention to the proof of Theorem~\ref{thm:main}.   As the proof has several parts we start with an outline.  In this section we show that:
\begin{itemize}
	\item a) $\then$ b) in Proposition~\ref{prop:Dsym->com}
	\item a) $\then$ c) in Proposition~\ref{prop:Dsym->right}.
\end{itemize} 
In  Subsection~\ref{subsec:sufficienty} we establish the following sufficient conditions for symmetry:
\begin{itemize}
	\item c) $\then$ e) and d) $\then$ e) in Proposition~\ref{prop:leftright->sym}
	\item b) $\then$ e) in Proposition~\ref{prop:com->sym}.
\end{itemize} 
In Subsection~\ref{subsec:Knuth} we finally show:
\begin{itemize}
	\item e) $\then$ a) in Proposition~\ref{prop:sym->Dsym}.
\end{itemize} 
Carrying out the above agenda shows that a), b), c), and e) are equivalent and that d) implies e). It remains to show that e) implies d).  Assuming, for the moment, that all but d) are equivalent we can give a short proof of this fact.  We do so next and then return to the agenda outlined above.

\begin{prop}
	If $B$ is symmetric then  $R_J^{-1}B \equiv R_K^{-1}B$ whenever $J\sim K$. 
\end{prop}
\begin{proof}
	If $B$ is symmetric then we know that it is $D$-commutative with all inverse $J$-classes and that $BR_J^{-1} \equiv BR_K^{-1}$ whenever $J\sim K$. Therefore
	$$R_J^{-1}B \equiv BR_J^{-1}\equiv BR_K^{-1} \equiv R_K^{-1}B$$
	whenever $J\sim K$.  
\end{proof}

We now begin with a proof that a) implies b).

\begin{prop}\label{prop:Dsym->com}
If $B\Subset \fS_n$ is $D$-symmetric then it is $D$-commutative with all inverse $J$-classes. 
\end{prop}
\begin{proof}
As $B$ is $D$-symmetric there exist bijections $\Psi_\cU:B\to B$ for each $\cU\in \Pi(n,J)$ that satisfy (\ref{eq:Dsym cond}).   It now follows by Lemma~\ref{lem:des char} that
\begin{align*}
u\in \Des(\rho_\cU(\pi)) &\iff u\in \Des(\delta_\cU) \cup (\Des(\pi)\cap \cU^*)\\ 
&\iff u\in \Des(\delta_\cU) \myor u \in \cU^*\myand \delta_\cU(u) \in \Des(\Psi_\cU(\pi))\\
& \iff u\in \Des(\sigma_\cU(\Psi_\cU(\pi))). 	
\end{align*}
Using this we obtain our desired result since
$$R_J^{-1}B =\bigsqcup_{\pi \in B\atop \cU\in \Pi(n,J)} \{\rho_\cU(\pi) \}
\equiv \bigsqcup_{\pi \in B\atop \cU\in \Pi(n,J)}\sigma_\cU(\Psi_\cU(\pi))
\equiv \bigsqcup_{\pi \in B\atop \cU\in \Pi(n,J)} \{\sigma_\cU(\pi)\} 	
\equiv BR_J^{-1},$$
where the first step follows from Lemma~\ref{lem:l as rho}, the second step follows our previous calculation, the third since $\Psi_\cU:B\to B$ is a bijection, and the last step because of Lemma~\ref{lem:r as sigma}.  
\end{proof}

Next we prove that  a) implies c). 

\begin{prop}\label{prop:Dsym->right}
If $B$ is $D$-symmetric, then $BR_J^{-1} \equiv BR_K^{-1}$ whenever $J\sim K$.  
\end{prop}
\begin{proof}
	Assume $B$ is $D$-symmetric and set $J=\{j_1<\cdots<j_p\}$. It suffices to prove the proposition when  $K$ is such that $\co(K)$ is the composition obtained by transposing the $k$th and $(k+1)$st blocks of $\co(J)$.  In particular, if we set $s' = r+t-s$ where 
$$r = j_{k-1},\ s = j_k,\ \myand\  t = j_{k+1},$$
then $K = J\setminus\{s\} \cup \{s'\}$.  Now define 
$$\cU = ([r], [s+1,t], [r+1,s], [t+1,n])$$
noting that 
$$\delta_\cU(u) = \begin{cases}
	u & \textrm{ if } u \in [r] \cup [t+1,n]\\
	u-(s-r) &\textrm{ if } u \in [s+1,t]\\
	u+(t-s) &\textrm{ if } u \in [t+1,s]
\end{cases}$$	

As $B$ is $D$-symmetric there exists a bijection $\Psi:B\to B$, corresponding to $\cU$, that satisfies (\ref{eq:Dsym cond}).  Letting $\pi\in B$ and $\tau = \Psi(\pi)$ it follows that 
	  $$\pi_1\ldots\pi_{r} \equiv \tau_1\ldots\tau_{r}\quad \myand \quad \pi_{t+1}\ldots\pi_{n} \equiv \tau_{t+1}\ldots\tau_{n}$$
that
	 $$\pi_{s+1}\ldots\pi_{t} \equiv \tau_{r+1}\ldots\tau_{s'}\quad\myand \quad \pi_{r+1}\ldots\pi_{s} \equiv \tau_{s'+1}\ldots\tau_{t}.$$
So Lemma~\ref{lem:shuf invariant} together with (\ref{eq:R_J as cshuf}) gives $\pi R_J^{-1} \equiv \Psi(\pi)R_K^{-1}$.  As $\Psi:B\to B$ is bijective the lemma now follows.  	
\end{proof}

\bigskip
\bigskip

To conclude this section we establish some needed properties of $D$-symmetric sets.  Our first lemma follows directly from the definition of $D$-symmetry.  In that lemma we make use of the following convention. For any finite set of positive integers $S$ set
$$\bx^S: =\prod_{i\in S}x_i.$$   
Our second lemma follows directly from the first lemma.  In both cases we omit formal proofs.

\begin{lemma}\label{lem:Dsym GFcond}
	Let $B\Subset \fS_n$. Then $B$ is $D$-symmetric if and only if for all $\cU\vdash [n]$ we have
\begin{equation}\label{eq:Dsym GFcond}
\sum_{\pi \in B} \bx^{\delta_\cU(\Des(\pi)\ \cap\ \cU^*)}  = \sum_{\pi \in B} \bx^{\Des(\pi)\ \cap\ \delta_\cU(\cU^*)}.
\end{equation}
\end{lemma}

\begin{lemma}\label{lem:Dsym prop}
Set $A,B\Subset \fS_n$.  We then have
\begin{enumerate}
\item[a)] If $B$ is $D$-symmetric and $B\equiv A$ then $A$ is also $D$-symmetric.
\item[b)] If $\{B_i\}_{i\in I}$ is a collection of $D$-symmetric multisets  then $\bigsqcup_{i\in I} B_i$ is $D$-symmetric.  
\item[c)] If $A,B$ are $D$-symmetric with $A\subseteq B$ then so is $B\setminus A$.
\end{enumerate}
\end{lemma}

\begin{lemma}\label{lem:Dsym reduction}
	A multiset $B\Subset \fS_n$ is $D$-symmetric if and only if for each $\cU=(U,V)\vdash [n]$ there exists a bijection $\Psi_\cU:B\to B$ satisfying (\ref{eq:Dsym cond}).  
\end{lemma}
\begin{proof}
The forward direction is trivial.  We concentrate on the reverse direction.  For the set partition $(\{1,2,\ldots, n\})$ consisting of 1 block, observe that the identity function on $B$ satisfies (\ref{eq:Dsym cond}).  Now assume, for an inductive proof, that for every set partition with $p$ blocks there exists a corresponding bijection that satisfies (\ref{eq:Dsym cond}).  Consider a set partition $\cU = (U_1,\ldots, U_p, U_{p+1})$ with $p+1$ blocks and define
$$\cV = (U_1,\ldots, U_{p-1}, U_p\cup U_{p+1}).$$
As $\cV$ has $p$ blocks we know by induction that there exists some bijection $\Psi_\cV:B\to B$ that satisfies (\ref{eq:Dsym cond}).  Define $\cW = ([n]\setminus W, W)\vdash [n]$ where $W = \delta_\cV(U_{p+1})$.  As this set partition has two blocks there exits a bijection $\Psi_\cW:B\to B$ satisfying (\ref{eq:Dsym cond}). It now suffices to prove that  $\Psi_\cW\circ \Psi_\cV$ is a bijection corresponding to $\cU$ that satisfies (\ref{eq:Dsym cond}).    

For $u\in \cU^*\subseteq \cV^*$ we know that $u,u+1$ are in the same block in $\cV$ and hence $\delta_\cV(u)+1 = \delta_\cV(u+1)$.  As the pair $u,u+1$ are also in the same block of $\cU$ then $u,u+1\in U_{p+1}$ or $u,u+1\in [n]\setminus U_{p+1}$.  So $\delta_\cV(u),\delta_\cV(u+1)$ are in the same block of $\cW$, i.e., $\delta_\cV(u)\in \cW^*$.  Combining these pieces it now follows that if $u\in \cU^*$ then 
\begin{align*}
u\in \Des(\pi) &\iff \delta_\cV(u)\in \Des(\Psi_\cV(\pi)) \\
&\iff \delta_\cW\circ \delta_\cV(u) \in \Des(\Psi_\cW\circ \Psi_\cV(\pi)) \\
&\iff \delta_\cU(u) \in \Des(\Psi_\cW\circ \Psi_\cV(\pi)),
\end{align*}
where the last equivalence follows by the easy-to-check fact that $\delta_\cU = \delta_\cW \circ \delta_\cV$.


\end{proof}

\subsection{Sufficiency}\label{subsec:sufficienty}

The goal of this subsection is to prove that each of b), c), and d) individually implies e).   We begin with some discussion and definitions.  Recall that 
\begin{equation}\label{eq:Q(B)2}
Q(B) = \sum_{\pi\in B} F_{\Des(\pi),n}  = \sum_{\alpha\models n}b_\alpha M_\alpha \in \QSYM(n)	
\end{equation}
for some integers $b_\alpha$.  Next recall the correspondence $\co$ between subsets and compositions and the notation $\alpha\geq \beta$, for $\alpha,\beta\models n$,  meaning that $\beta$ is a refinement of $\alpha$.  Now for each $J\subseteq [n-1]$  the definition of the fundamental basis means that
$$b_{\co(J)} = |\makeset{\pi\in B}{\co(\Des(\pi)) \geq \co(J)}| = |\makeset{\pi\in B}{\Des(\pi) \subseteq J}|.$$
Defining $B_J:= \makeset{\pi\in B}{\Des(\pi)\subseteq  J}$ we have $b_{\co(J)} = |B_J|$. Recall that the monomial symmetric function $m_\lambda$  can be written as 
$$m_\lambda = \sum_{\alpha\sim \lambda} M_\alpha.$$ 
Therefore to show $Q(B) \in \Lambda(n)$ it suffices to prove that $|B_J| = |B_K|$ whenever $J\sim K$.

With this discussion in mind we turn to the proof that c) and d) each imply e).  

\begin{prop}\label{prop:leftright->sym}
If $BR_J^{-1} \equiv BR_K^{-1}$ whenever $J\sim K$ or $R_J^{-1}B \equiv R_K^{-1}B$ whenever $J\sim K$, then  $B$ is symmetric.  
\end{prop}
\begin{proof}
We start with the first claim. By Lemma~\ref{lem:r as sigma} we have
	$$\sum_{\pi \in BR_J^{-1}} \bx^{\Des(\pi)} = \sum_{\cU\in \Pi(n,J)\atop \pi\in B}  \bx^{\Des(\sigma_\cU(\pi))}.$$
Now observe that the coefficient on $\bx^\emptyset$ in this expression is
	$$|\makeset{\pi\in B}{\Des(\pi)\subseteq J}| = |B_J|.$$
This can be seen by considering (\ref{eq:des sigma}) of Lemma~\ref{lem:des char} and noting that in order for $\Des(\sigma_\cU(\pi)) = \emptyset$ we must have $\delta_\cU = 1$. By our assumption we know that if $J\sim K$ then $BR_J^{-1} \equiv BR_K^{-1}$.  So $|B_J| = |B_K|$ which, in light of discussion above, proves our first claim.  

To prove the second claim, we know by Lemma~\ref{lem:l as rho} that 
	$$\sum_{\pi \in R_J^{-1}B} \bx^{\Des(\pi)} = \sum_{\cU\in \Pi(n,J)\atop \pi\in B}  \bx^{\Des(\rho_\cU(\pi))}.$$
By appealing to (\ref{eq:des rho}) in Lemma~\ref{lem:des char}, a similar proof to that in the first case establishes our second claim.
\end{proof}

We now turn our attention to proving that if $B$ is $D$-commutative with all inverse $J$-classes then it is symmetric, i.e., that b) implies e) in our main theorem.  For reference and to set the stage note that Lemmas~\ref{lem:r as sigma} and \ref{lem:l as rho} imply that if $BR_J^{-1} \equiv R_J^{-1}B$ then 
\begin{equation}\label{eq:B commutes}
\sum_{\cU\in\Pi(n,J)\atop \pi\in B} \textbf{x}^{\Des(\sigma_\cU(\pi))}= \sum_{\cU\in\Pi(n,J)\atop \pi\in B} \textbf{x}^{\Des(\rho_\cU(\pi))}.
\end{equation}
They key idea in the coming proofs is to consider the coefficient $c_i$ on $\bx^{\{i\}}$ in this generating function.  To describe this coefficient we make the following definitions.

\begin{definitions}
	For any $\cU\in \Pi(n)$ define 
$$r(\cU):= [n-1]\setminus \cU^*\quad \myand\quad s(\cU):=[n-1]\setminus \delta_\cU(\cU^*).$$
Additionally, for any $\alpha\models n$ and $i\in [n-1]$ set
	$$\Pi_i(\alpha) = \makeset{\cU\in \Pi(\alpha)}{\Des(\delta_\cU) = \{i\}}$$
and let $\Gamma_i(\alpha)$ consist of all $\cU\in \Pi_i(\alpha)$ such that all blocks in $\cU$ are intervals. 
\end{definitions}

We now have the following lemma.

\begin{lemma}\label{lem:coef on i}
Fix $\alpha\models n$.  Take $f_\cU= \rho_\cU$ or $\sigma_\cU$ and define $c_i$ to be the coefficient on  $\bx^{\{i\}}$ in 
$$\sum_{\cU\in \Pi(\alpha)\atop \pi\in B} \bx^{\Des(f_\cU(\pi))}.$$
Then there is some $a\geq 0$ so that
$$c_i = \begin{cases}
	a+ \sum_{\cU\in \Pi_i(\alpha)}|B_{r(\cU)}|& \textrm{ if } f_\cU = \sigma_\cU\\
	a+ \sum_{\cU\in \Pi_i(\alpha)}|B_{s(\cU)}|& \textrm{ if } f_\cU = \rho_\cU.
\end{cases}
$$
\end{lemma}

\begin{proof}
First consider the case when $f_\cU=\rho_\cU$.  	Appealing to (\ref{eq:des rho}) in Lemma~\ref{lem:des char} we have $\Des(\rho_\cU(\pi))=\{i\}$ if and only if 
$$\Des(\delta_\cU)=\emptyset \quad\myand\quad  \Des(\pi)\cap \cU^* = \{i\}$$
or vice versa.  In the displayed case, note that $\Des(\delta_\cU)=\emptyset$ if and only if $\cU$ is the unique partition $\cI\in \Pi(\alpha)$ whose first block is the interval $[\alpha_1]$ and whose second block is the following $\alpha_2$ integers, etc.  Set $a = |\makeset{\pi\in B}{\Des(\pi)\cap \cI^*}|$.
Now consider the other case when
$$\Des(\delta_\cU)=\{i\} \quad\myand \quad  \Des(\pi)\cap \cU^* = \emptyset.$$
This occurs if and only if $\cU\in \Pi_i(\alpha)$ and $\Des(\pi)\subseteq [n-1]\setminus \cU^* = r(\cU)$.  This establishes the case when $f_\cU = \rho_\cU$.  

Now consider the case when  $f=\sigma_\cU$.  Appealing to (\ref{eq:des sigma}) in Lemma~\ref{lem:des char} we have $\Des(\sigma_\cU(\pi)) = \{i\}$ if and only if 
$$\Des(\delta_\cU) = \emptyset \quad \myand\quad  \makeset{u\in \cU^*}{\delta_\cU(u)\in\Des(\pi)} = \{i\}$$
or vice versa.  As before, the displayed case can only occur when $\cU = \cI$.  As $\delta_\cI = 1$ we further have 
$$\{i\} = \makeset{u\in \cI^*}{\delta_\cI(u)\in\Des(\pi)} = \Des(\pi)\cap \cI^*.$$
so that we can take $a$ as above in this case as well.  Now consider the case when 
$$\Des(\delta_\cU) = \{i\}  \quad \myand\quad  \makeset{u\in \cU^*}{\delta_\cU(u)\in\Des(\pi)} = \emptyset.$$
This occurs when $\cU\in \Pi_i(\alpha)$ and $\Des(\pi) \cap \delta_\cU(\cU^*) = \emptyset$. As the latter is equivalent to $\Des(\pi) \subseteq [n-1] \setminus \delta_\cU(\cU^*)\subseteq s(\cU)$,   this explains our second term.  

\end{proof}

\begin{lemma}\label{lem:r and s perms}
	Fix $\alpha\models n$ For any $\cU\vdash \Pi(\alpha)$ we have $r(\cU) \sim s(\cU)$.  
\end{lemma}
\begin{proof}
First consider the case when all the blocks of $\cU$ are intervals.  Let $J\subseteq[n-1]$ be such that $\co(J) = \alpha$. As the blocks in $\cU$ are intervals we see that
$$\delta_\cU(\cU^*) = [n-1]\setminus J.$$ 
This means that $s(\cU)=J$.    We must now show that $\co(r(\cU)) \sim \alpha = \co(J)$.   Again using the fact that the blocks of $\cU$ are intervals we may define $\cW=(W_1, W_2,\ldots)$ to be the ordered set partition obtained by permuting the blocks of $\cU$ so that $\max(W_i) < \min(W_{i+1})$.    As $\cW^* = \cU^*$ it follows that
$$r(\cU) = r(\cW)=\{\max(W_1)<\max(W_2)<\cdots\}$$  
which, in turn, implies that $\co(r(\cU)) = (|W_1|, |W_2|, \ldots)$. By our choice of $\cW$ our claim now follows.

Now consider an arbitrary $\cU\in \Pi(\alpha)$ and let $\cV$ be the refinement given by replacing each block $U_i$ of $\cU$ with the sequence $(I_1,I_2,\ldots)$ of maximal nonempty intervals in $U_i$ ordered so that $\max(I_i)< \min(I_{i+1})$.    Observe that $\cV^* = \cU^*$ and $\delta_\cU = \delta_\cV$. Consequently, $r(\cU) = r(\cV)$ and $s(\cU) = s(\cV)$.  The general claim now follows from our first paragraph.    
\end{proof}

For the next few proofs, some additional terminology relating to permutations of compositions is required. 
For any composition $\alpha\models n$ with $p$ parts and $I=\{i_1<i_2<\cdots <i_s\}\subseteq [p]$ we write
$$\alpha(I):= (\alpha_{i_1},\alpha_{i_2},\ldots, \alpha_{i_s},\alpha_{j_1},\alpha_{j_2},\ldots, \alpha_{j_t})$$
where $j_1<j_2<\ldots <j_t$ are the elements in $[p]\setminus I$. In particular for any $m\leq p$ we have $\alpha([m]) = \alpha$ and, in general, $\alpha\sim \alpha(I)$.  We also define $S_k(\alpha)$ to be the set of all $I\subseteq[p]$ such that $\sum_{i\in I}\alpha_i = k$.  As all the parts of $\alpha$ are positive integers there is at most one $I=[m]\subseteq [p]$ with $I\in S_k(\alpha)$.  In this case set  $S_k'(\alpha) = S_k(\alpha)\setminus\{I\}$ otherwise  set $S_k'(\alpha) = S_k(\alpha)$.

\begin{lemma}\label{lem:coef equal}
	Fix $\lambda\vdash n$ and assume for each composition $\alpha\sim \lambda$ there exists some $C_\alpha\geq 0$. If for each such $\alpha$ and $k\leq n$ we have
\begin{equation}\label{eq:constance}
	|S_k(\alpha)|\cdot C_\alpha = \sum_{I\in S_k(\alpha)} C_{\alpha(I)},
\end{equation}
then the $C_\alpha$'s are equal.
\end{lemma}

\begin{proof}
For $\alpha,\beta\sim \lambda$ there exists a sequence of compositions 
$$\alpha=\alpha^{(1)}, \alpha^{(1)}, \ldots, \alpha^{(t)}=\beta$$
so that $\alpha^{(i+1)} = \alpha^{(i)}(\{j\})$ for some $j$.  Now choose $\alpha$ so that $C_\alpha$ is maximized.  By this choice of $\alpha$ together with (\ref{eq:constance}) and $k = \alpha_j$ it follows that $C_\alpha = C_{\alpha(\{j\})}$.  This with our first observation implies our lemma.   
\end{proof}

\begin{lemma}\label{lem:Gamma bij}
	Set $k< n$ and  $\alpha\models n$.   There exists a bijection 
	$$f:S_k'(\alpha)\to \Gamma_k(\alpha)$$
	so that for each $I\in S_k'(\alpha)$ we have $\co(r(f(I))) = \alpha(I)$.   
\end{lemma}
\begin{proof}
Fix $I \in S_k'(\alpha)$ and let $J\subseteq[n-1]$ be such that $\co(J) = \alpha(I)$. Define $f(I)=(U_1,\ldots, U_p)\vdash [n]$ by
$$(U_i)_{i\in I} = ([j_1], [j_1+1,j_2],\ldots, [j_{t-1},j_{t}]) \vdash [k]$$
and 
$$(U_i)_{i\in [p]\setminus I}= [j_{t+1},j_{t+2}],\ldots, [j_{p-1},n]) \vdash [n]\setminus [k].$$
As $I \neq [m]\subseteq [p]$ for some form $m\leq p$ it follows that $f(I)\in \Gamma_k(\alpha)$.  Our map $f$ is certainly injective.  It also follows that every partition in $\Gamma_k(\alpha)$ can be constructed in this manner.  Hence $f$ is also surjective.  

Continuing with the notation above we see that $r(f(I)) = [n-1]\setminus f(I)^* = J$.  As $\co(J) = \alpha(I)$ this justifies our last claim.

%
%
\end{proof}

\begin{prop}\label{prop:com->sym}
	If $B\Subset \fS_n$ is $D$-commutative with all inverse $J$-classes then $B$ is symmetric.  
\end{prop}
\begin{proof}
Define $B_\alpha:= B_J$ where $\co(J) = \alpha$.  By recalling the discussion at the start of this subsection, it suffices to prove this proposition by showing that  $|B_\alpha| = |B_\beta|$  whenever $\alpha\sim \beta$.  We proceed by induction on the number of parts in our compositions where the base case is when our composition has $n$ parts.  This case holds trivially since $(1^n)$ is the only composition with $n$ parts.  

Take $\alpha\models n$ with $p$ parts.  As $B$ is $D$-commutative with all inverse $J$-classes, we know that (\ref{eq:B commutes}) holds for this $\alpha$.  In light of this equation and Lemma~\ref{lem:coef on i} it follows that 
$$\sum_{\cU\in \Pi_k(\alpha)}|B_{r(\cU)}|=\sum_{\cU\in \Pi_k(\alpha)}|B_{s(\cU)}|$$
holds for all $1\leq k<n$.  Now consider a particular  $\cU\in \Pi_k(\alpha)\setminus \Gamma_k(\alpha)$. As $\cU$ has $p$ blocks and at least one of them is not an interval, it follows that $|[n-1]\setminus \cU^*|\geq p$.  Therefore the composition corresponding to $r(\cU)$ and $s(\cU)$ has at least $p+1$ parts and by Lemma~\ref{lem:r and s perms} we know that $r(\cU) \sim s(\cU)$.    It now follows by induction that $|B_{r(\cU)}|= |B_{s(\cU)}|$.  
Consequently 
$$\sum_{\cU\in \Gamma_k(\alpha)}|B_{r(\cU)}| = \sum_{\cU\in \Gamma_k(\alpha)}|B_{s(\cU)}| = |\Gamma_k(\alpha)|\cdot |B_\alpha|,$$
where the second equality follows from the fact that if $\cU\in \Gamma_k(\alpha)$ then $\co(s(\cU)) = \alpha$. By appealing the bijection in Lemma~\ref{lem:Gamma bij} we have for all $k< n$
\begin{equation}\label{eq:gamma2}
 \sum_{I\in S_k'(\alpha)} |B_{\alpha(I)}| = |S_k'(\alpha)|\cdot |B_\alpha|.
\end{equation}
In the case $I \in S_k(\alpha)\setminus S_k'(\alpha)$ then $\alpha(I) = \alpha$.  By adding the term $|B_\alpha|$ to both sides of (\ref{eq:gamma2}) if necessary we have
\begin{equation}\label{eq:gamma2}
 \sum_{I\in S_k(\alpha)} |B_{\alpha(I)}| = |S_k(\alpha)|\cdot |B_\alpha|.
\end{equation}
for all $k<n$ and compositions $\alpha$ with $p$ parts.  As the case when $k=n$ yields a trivial equality we may further assume $k\leq n$.  Appealing to Lemma~\ref{lem:coef equal} with $C_\alpha := |B_\alpha|$ we conclude that $|B_\alpha| = |B_\alpha|$ for compositions $\alpha\sim \beta$ with $p$ parts.  This completes our proof.  
\end{proof}

\subsection{Necessity: Symmetry implies $D$-symmetry}\label{subsec:Knuth}

The goal of this subsection is to prove that symmetric multisets are $D$-symmetric, i.e., to prove Proposition~\ref{prop:sym->Dsym}.   We first show that this proposition holds for fine sets and, using this fact, ``bootstrap'' up to the general result.  As fine sets are intimately connected to the theory of tableaux we begin by introducing the needed ideas from this theory. 

A \emph{standard Young tableaux} of shape $\mu\vdash P$ is a filling of the Young digram of $\mu$ with each number in $[n]$ used exactly once so that rows and columns are strictly increasing.  We denote by $\SYT(\mu)$ the set of all standard Young tableaux of shape $\mu$ and set $\SYT(n) = \cup_{\mu\vdash n}\SYT(\mu)$.  For any $P\in \SYT(n)$ and $m\leq n$ we define $P_{<m}$ to be the standard Young tableaux in $\SYT(m-1)$ given by the entries in $P$ that are $<m$.    We refer to a coordinate location in a Young tableau as a \emph{box} and the element in a box as a \emph{value}.  All boxes are coordinatized using matrix coordinates.

For any $P\in \SYT(n)$ we define its \emph{descent set} by
$$\Des P: = \makeset{i}{\textrm{$i+1$ is on a row below $i$ in $P$}}\subseteq [n-1].$$

\ytableausetup{mathmode, boxsize=2.3em}
Given $P\in\SYT(n)$ we say a sequence of boxes $b_0,\ldots, b_m$ in $P$ is a \emph{promotion path} provided that $b_{i+1}$ is whichever of the boxes immediately below or to the right of $b_i$ that contain the smaller value for all $0\leq i<m$.  So given an initial box $b_0$ the maximal promotion path starting at $b_0$ is uniquely determined.  Consequently it makes sense to define the \emph{$v$-promotion path} in $P$ to be the maximal promotion path that starts at the box containing the value $v$.  

For any $\mu\vdash n$ and $a,b\in [n]$ with $a\leq b$ define the \emph{promotion operator}
$$\partial_a^b: \SYT(\mu) \to \SYT(\mu)$$ 
as follows.  Fix some $P\in \SYT(\mu)$ and consider the \emph{skew tableau} formed by the values in $[a,b]$.    Let $c_0,\ldots, c_m$ be the $a$-promotion path in this skew tableau.  Next delete the entry in box $c_0$ and slide the value in $c_{i+1}$ into $c_i$.  As $c_m$ is now empty, place in it the value $b+1$.  Finally decrement each value in this skew tableau by 1 to obtain $\partial_a^b P\in \SYT(\mu)$. 

\begin{exa}
Take $P$ to be the tableau on the left then $\partial_3^{12}P$ is the tableau on the right:
	\ytableausetup{mathmode, boxsize=2.2em}
$$\begin{ytableau}
1& *(gray)3& *(lightgray)6& *(lightgray)7\\
2& *(gray)5& *(gray)9& *(gray)11\\
*(lightgray)4& *(lightgray)10& 13& 15\\
*(lightgray)8& 14\\
*(lightgray)12
\end{ytableau}
\qquad\qquad \qquad 
\begin{ytableau}
1& *(lightgray)4& *(lightgray)5& *(lightgray)6\\
2& *(lightgray)8& *(lightgray)10& *(lightgray)12\\
*(lightgray)3& *(lightgray)9& 13& 15\\
*(lightgray)7& 14\\
*(lightgray)11
\end{ytableau}\ .$$
Here the boxes corresponding to the skew tableau are in gray and the promotion path in $P$ is in dark gray.
\end{exa}
We point out two important properties of the promotion operator.  First it is clear that
$$P_{<a}=(\partial_a^b P)_{<a}.$$
Second, we see that for each $\mu\vdash n$ the mapping $\partial_a^b:\SYM(\mu) \to \SYM(\mu)$ is bijective as its inverse can be constructed as follows.  Let $c_0$ be the box containing $b$ and define the unique maximal sequence of boxes $c_0,\ldots, c_m$ so that $c_{i+1}$ is whichever of the boxes immediately above or to the left of $c_i$ that contains the larger value.   Now delete the value $b$ in $c_0$ and slide the value in $b_{i+1}$ into box $b_{i}$.  Next increment all the values in $[a,b]$ by 1 and place $a$ in the empty box $b_m$. 

We introduce more theory related to tableaux below as it is needed.  For now we have enough to prove our first few lemmas.

\begin{lemma}\label{lem:u iff u-1}
	For $Q\in \SYT(n)$ and $a<u<b$ we have 
$$u\in \Des Q\ \Longleftrightarrow\ u-1 \in \Des(\partial_a^b Q).$$
\end{lemma}
\begin{proof}
	First assume that $u\in \Des Q$.  In the calculation of $\partial_a^b Q$ the values in $[a+1,b]$ in $Q$  shift up one unit, shift left one unit, or remain fixed before being decremented by 1. Hence the only way $u-1\notin \Des(\partial_a^b Q)$ is if $Q$ contains one of the following configurations:
\begin{center}
\ytableausetup{mathmode, boxsize=2.2em}
\begin{ytableau}
*(lightgray) w &z_1 & \cdots & z_\ell &  u \\
*(lightgray) u{+}1
\end{ytableau}
\quad\myor\quad
\begin{ytableau}
*(lightgray)w & *(lightgray)u \\
z_1 & *(lightgray)u{+}1
\end{ytableau}\ ,
\end{center}
where the promotion path is highlighted in gray.  In both cases a simple check shows that such promotion paths are impossible. (E.g., in the second case $z_1<u$.) So neither of these two cases can occur.   We conclude that the forward direction of our lemma holds.

To prove the reverse direction, we need to show that if $u\notin\Des Q$ then $u-1\notin\Des(\partial_a^bQ)$.  Observe that if $Q^*$ denotes the conjugate of $Q$ then $u\in \Des Q^*$.  Also note that the operation of taking the conjugate commutes with $\partial_a^b$.  So to prove this direction we need only apply the above argument to $Q^*$.  The lemma now follows.  
\end{proof}

\begin{lemma}\label{lem:v,v+1}
Let $Q\in \SYT(n)$ and $u\leq k+1<n$.    Then
$$u\in \Des Q\iff k+1\in \Des(\partial_{u}^{k+1}\circ \partial_{u+1}^{k+2} Q ).$$
\end{lemma}

\begin{proof}
Set $\partial:= \partial_{u}^{k+1}\circ \partial_{u+1}^{k+2}$.  Suppose for a contradiction that $u\in \Des(Q)$ and $k+1\notin \Des(\partial Q)$.   Let $B=(b_0,\ldots, b_m)$ be the promotion path used by $\partial_{u+1}^{k+2}$ and let $C=(c_0,\ldots,c_\ell)$ be the promotion path used by $\partial_u^{k+1}$.  So in $Q$ boxes $b_0$ and $c_0$ contain $u+1$ and $u$ respectively, and in $\partial Q$ boxes $b_m$ and $c_\ell$ contain $k+2$ and $k+1$ respectively.   As $u\in \Des Q$ then $c_0$ is strictly above $b_0$.  By our assumption that $k+1\notin\Des(\partial Q)$ we must also have that  $c_\ell$ is weakly below $b_m$.  In fact since $b_m$ contains $k+2$ in $\partial_{u+1}^{k+2} Q$ then $c_\ell$ is also strictly to the left of $b_m$.      Now consider the first time a box in $C$ is weakly below and strictly to the left of some box in $B$.  Certainly this box cannot be $c_0$ (since $c_0$ is above $b_0$) and in fact we must have the following configuration of values in $Q$:
\ytableausetup{mathmode, boxsize=2.2em}
\begin{center}
\definecolor{red}{HTML}{f12e2e}
\definecolor{blue}{HTML}{204ae2}
\definecolor{purple}{HTML}{893C88}
\begin{ytableau}
*(blue) & w \\
*(purple) & *(red)u
\end{ytableau}\ ,
\end{center}
where the blue and purple boxes are in $C$ and the red and purple boxes are in $B$. So $w<u$. Therefore in $\partial_{u+1}^{k+2}Q$ the purple box contains $u-1$ and the white box contains $w' =w$ or $w-1$.  Since the standard Young tableaux $\partial_{u+1}^{k+2}Q$ contains exactly one occurrence of $w$ it follows that $u-1>w'$.  It is now immediate that if the blue box is in $C$ then $C$ must contain the white box and not the purple box.   We arrive at our desired contradiction proving that the forward direction holds.

The converse can be proved by an argument similar to that used to prove the converse in the previous lemma.  
\end{proof}

\begin{definition}
For any $\mu\vdash n$ let $k\leq n$ be such that $V=\{v_1<v_2<\cdots< v_{n-k}\}\subseteq [n]$. Define the mapping $\partial_V:\SYT(\mu) \to \SYT(\mu)$
by 
$$\partial_{V}:= \partial_{v_1}^{k+1}\circ\partial_{v_2}^{k+2}\circ \cdots\circ \partial_{v_{n-k}}^n.$$ 
Take $\partial_{\emptyset}$ to be the identity function. 
\end{definition}
We pause to comment on our choice of indexing above.  In what follows the operator $\partial_V$ appears in the context where $V$ is the second block of an ordered partition where the first block has size $k$. As a result we have chosen to index the elements of $V$ as above. 

\begin{exa}
	For example $V= \{3,9,10\}\subseteq [15]$.  Then we have
\begin{center}
\ytableausetup{mathmode, boxsize=1.9em}
\begingroup
\setlength{\tabcolsep}{10pt} 
\renewcommand{\arraystretch}{2} 
	\begin{tabular}{cccc}
$P$& $P_1=\partial_{10}^{15}P$ & 	$P_2=\partial_{9}^{14}P_1$ & $\partial_V P=\partial_{3}^{13}P_2$ \\
$\begin{ytableau}
1& 3& 6& 7\\
2& 5& 9& *(lightgray)11\\
4& *(lightgray)10& *(lightgray)13& *(lightgray)15\\
8& *(lightgray)14\\
*(lightgray)12
\end{ytableau}$
&
$\begin{ytableau}
1& 3& 6& 7\\
2& 5& *(lightgray)9& *(lightgray)10\\
4& *(lightgray)12& *(lightgray)14& 15\\
8& *(lightgray)13\\
*(lightgray)11
\end{ytableau}$
&
$\begin{ytableau}
1& *(lightgray)3& *(lightgray)6& *(lightgray)7\\
2& *(lightgray)5& *(lightgray)9& 14\\
*(lightgray)4& *(lightgray)11& *(lightgray)13& 15\\
*(lightgray)8& *(lightgray)12\\
*(lightgray)10
\end{ytableau}$
&
$\begin{ytableau}
1& 4& 5& 6\\
2& 8& 12& 14\\
3& 10& 13& 15\\
7& 11\\
9
\end{ytableau}$
\end{tabular}
\endgroup
\end{center}
where the skew tableau used to obtain the following tableau is highlighted.  
\end{exa}

The next lemma follows immediately from the fact that each promotion operator is bijective. 

\begin{lemma}\label{lem:partial bij}
Fix $\mu\vdash n$ and some $V\subseteq [n]$.  Then $\partial_V:\SYT(\mu) \to \SYT(\mu)$ is a bijection.
\end{lemma}

\begin{lemma}\label{lem:Knuth Dsym}
Let $Q\in \SYT(n)$ and fix $\cU=(U,V)\in \Pi(n)$.  For $u\in \cU^*$ we have
$$u\in \Des Q\iff \delta_\cU(u)\in \Des(\partial_{V} Q).$$
\end{lemma}

\begin{proof}
We prove this by induction on $|	V|$.  Observe that when $V=\emptyset$ then $\partial_VQ = Q$ and $\delta_\cU = 1$ so the lemma holds in this case.  

Now take $|U|=k<n$ so that $|V|>0$ and let $v_1 = \min(V)$.  Define $V'=V\setminus \{v_1\}$ and $U' = U \cup \{v_1\}$ and set $\cW = (U',V')$. For $x\in [v_1]$ we then have $\delta_\cW(x) = x$ and by induction we know that if $u\in \cW^*$ then
$$u\in \Des Q \iff \delta_\cW(u) \in \Des( \partial_{V'} Q).$$
Now define the function $f:[n]\to [n]$ by setting
$$f(x) = \begin{cases}
	x - 1 & \textrm{ if } v_1<x\leq k+1\\
	x & \textrm{ otherwise}.
\end{cases}$$
Observe that for $x\neq v_1$ we have $f\circ\delta_\cW(x) =\delta_\cU(x)$. 

Next recall that the positions of values less than $v_1$ and greater than $k+1$ are invariant under the action of $\partial_{v_1}^{k+1}$.  This together with  Lemma~\ref{lem:u iff u-1} tells us that for any $T\in\SYT(n)$ and $u\neq v_1-1,v_1,k+1$ we have
\begin{align}
	u\in \Des T &\iff f(u) \in \Des(\partial_{v_1}^{k+1}T).
\end{align}
We now consider the following two cases.

\medskip
\noindent
\textbf{Case 1: } $u\in \cU^*$ and $u\neq v_1$ 
\medskip

As $u\neq v_1$ then $\delta_\cW(u) \neq v_1$. As $v_1=\min(V)$ then $u\neq v_1-1$ and so $\delta_\cW(u) \neq v_1-1$.    As $u\neq \max(U)$ then $\delta_\cW(u) \neq k+1$.  Computing we now have
\begin{align*}
u\in \Des Q &\iff \delta_\cW(u) \in \Des( \partial_{V'} Q)\\
& \iff f\circ\delta_\cW(u) \in \Des( \partial_{v_1}^{k+1}\partial_{V'} Q)\\
& \iff \delta_\cU(u) \in \Des( \partial_V Q).
\end{align*}
So the lemma holds in this case.

\medskip
\noindent
\textbf{Case 2: } $u=v_1\in \cU^*$
\medskip

In this case $|V|\geq 2$ and $v_1+1 = v_2$.  Set $V'' = V\setminus \{v_1,v_2\}$ so that $V''$ is either empty or all its values are $>v_2$.  By definition we have
$$\partial_V Q  = \partial_{v_1}^{k+1}\circ \partial_{v_2}^{k+2}R$$
where $R = \partial_{V''}Q$. (Note that if $|V| = 2$ then $Q=R$.)  As  $Q_{<v_2+1} = R_{<v_2+1}$ and  $v_1<v_2$ we have $v_1\in\Des Q \iff v_1\in \Des R$. This together with Lemma~\ref{lem:v,v+1} gives 
\begin{align*}
v_1\in \Des Q &\iff v_1 \in \Des R \\
&\iff k+1\in \Des(\partial_V Q) \\
&\iff \delta_\cU(v_1)\in \Des(\partial_V Q).
\end{align*}
This completes our proof.

\end{proof}

In order to continue we recall some additional well known results from the theory of tableaux.  We denote the \emph{Robinson-Schensted correspondence} by 
$$\RS:\fS_n \to \bigcup_{\mu\vdash n} \SYT(\mu)\times \SYT(\mu).$$  
If $\pi\mapsto (P,Q)$ under this bijection we call $P$ the \emph{insertion tableau}. We omit the definition of $\RS$ as it is not needed here but we point out two key properties of this mapping that are well-known.  First $\RS$ is a bijection.  The second well known property we state as a lemma. For a reference see \cite{SaganBook}, Chapter 5, Exercise 22.

\begin{lemma}\label{lem:des under RS}
	If $\RS(\pi) = (P,Q)$, then $\Des(\pi) = \Des Q$.  
\end{lemma}

For each $P\in \SYT(n)$ we define the corresponding \emph{Knuth class} to be the set
$$K(P):= \makeset{\pi\in \fS_n}{\textrm{$P$ is the insertion tableau for $\pi$}}.$$
Additionally we need the following well-known result due to Gessel.

\begin{thm}[\cite{Ges84}]\label{thm:ges}
For any $\lambda\vdash n$ we have
$$s_\lambda = \sum_{Q\in\SYT(\lambda)} F_{\Des Q,n}.$$
\end{thm}

This result together with the Robinson-Schensted correspondence tells us that if $P$ has shape $\lambda$ then $Q(K(P)) = s_\lambda$.  In fact we can say more. Consider any fine multiset $B\Subset\fS_n$ with $Q(B) = \sum_{\lambda\vdash n} c_\lambda s_\lambda$ where $c_\lambda\geq 0$.  For each shape $\lambda$ fix some $P_\lambda$ of that shape.  Define the multiset 
$$A = \bigsqcup_{\lambda\vdash n}\bigsqcup_{i=1}^{c_\lambda} K(P_\lambda)$$
so that $Q(A) = Q(B)$ and hence $A\equiv B$. We make use of this in our next lemma.

\begin{lemma}\label{lem:fine->Dsym}
	If $B\Subset\fS_n$ is fine then it is $D$-symmetric.
\end{lemma}

\begin{proof}

As $B$ is fine it follows from the discussion preceding the statement of this lemma that $B\equiv  \bigsqcup K(P_i)$ where our disjoint union is over some finite index set $I$ and $P_i\in \SYT(n)$. By  parts a) and b) of Lemma~\ref{lem:Dsym prop} it then suffices to show that individual Knuth classes are $D$-symmetric.  

To show this fix $P\in \SYT(\mu)$ and consider the Knuth class $K(P)$.  Take $\cU=(U,V)\vdash [n]$ and consider the mapping $\Psi_\cU:K(P)\to K(P)$ given by 
$$\pi \mapsto \RS^{-1}(P,\partial_V Q)$$
where $RS(\pi) = (P,Q)$.  The bijectivity of this mapping follows from Lemma~\ref{lem:partial bij} and the fact that $\RS$ is bijective.  Now take $u\in \cU^*$ and fix some $\pi\in K(P)$ with $\RS(\pi) = (P,Q)$.  By Lemma~\ref{lem:des under RS}  and Lemma~\ref{lem:Knuth Dsym}, since $\cU=(U,V)$, we obtain our desired result
$$u\in \Des(\pi) \iff u\in \Des Q \iff \delta_\cU(u) \in \Des(\partial_V Q)\iff \delta_\cU(u) \in \Des(\Psi_\cU(\pi)).$$
By Lemma~\ref{lem:Dsym reduction} we now conclude that $B$ is $D$-symmetric.
\end{proof}

\begin{prop}\label{prop:sym->Dsym}
If $B\Subset \fS_n$ is symmetric then it is $D$-symmetric.  	
\end{prop}
\begin{proof}
	As $B$ is assumed to be symmetric and the Schur functions $\{s_\lambda\}_{\lambda\vdash n}$ are an integral basis for $\Lambda(n)_\mathbb{Z}$, the ring of symmetric functions of degree $n$ with integral coefficients,  we can write
$$Q(B) = \sum_{\lambda\vdash n} c_\lambda s_\lambda$$
where $c_\lambda\in \mathbb{Z}$.  Now define the multiset
$$A = \bigsqcup_{\lambda\vdash n\atop c_\lambda<0}\bigsqcup_{i = 1}^{|c_\lambda|} K(\lambda)$$
so that by Theorem~\ref{thm:ges} have
$$Q(A) = \sum_{\lambda\vdash n \atop c_\lambda<0}|c_\lambda|s_\lambda,$$ 
so that $A$ is fine, and by our choice $A$ we also see that
$$Q(A\sqcup B) = Q(A) + Q(B) = \sum_{\lambda\vdash n\atop c_\lambda>0} c_\lambda s_\lambda,$$
so that $A\sqcup B$ is also fine. By Lemma~\ref{lem:fine->Dsym} we now see that $A$ and $A\sqcup B$ are both $D$-symmetric.  It now follows from Lemma~\ref{lem:Dsym prop} that $B$ is $D$-symmetric as well.  
\end{proof}

\bibliographystyle{alpha}
\bibliography{mybib}

\end{document}